\documentclass[11pt]{article}
\usepackage{epigamath}

\usepackage[notext]{kpfonts}
\usepackage{baskervald}

\setpapertype{A4}

\usepackage[english]{babel}

\usepackage[dvips]{graphicx}     

\removelength{1cm}

\title{$\mathbb{Q}_\ell$-cohomology projective planes and Enriques \\ surfaces in characteristic two}
\titlemark{$\mathbb{Q}_\ell$-cohomology projective planes (characteristic two)}
\author{Matthias Sch\"utt}
\authoraddresses{
\authordata{Matthias Sch\"utt}{\firstname{Matthias} \lastname{Sch\"utt}\\
\institution{Institut f\"ur Algebraische Geometrie, Leibniz Universit\"at Hannover, Welfengarten 1, 30167 Hannover, Germany \\
Riemann Center for Geometry and Physics, Leibniz Universit\"at Hannover, Appelstrasse 2, 30167 Hannover, Germany}\\
\email{schuett@math.uni-hannover.de}}
}
\authormark{M. Sch\"utt}
\date{\vspace{-5ex}} 
\journal{\'Epijournal de G\'eom\'etrie Alg\'ebrique} 
\acceptation{Received by the Editors on October 12, 2017, and in final form on May 2, 2019. \\ Accepted on June 1, 2019.}

\newcommand{\articlereference}{Volume 3 (2019), Article Nr. 10}

\acknowledgement{Partial funding by ERC StG~279723 (SURFARI) is gratefully acknowledged.}

\usepackage[all]{xy}

\allowdisplaybreaks

\newdimen\origiwspc
\origiwspc=\fontdimen2\font

\usepackage{float}

\usepackage{enumitem}
\setlist[enumerate,1]{label={\rm (\arabic*)}}

\newcommand{\C} {\mathbb{C}}
\newcommand{\Q} {\mathbb{Q}}

\newcommand{\F}{\mathbb{F}}
\newcommand{\Z}{\mathbb{Z}}

\newcommand{\PP}{\mathbb{P}}
\newcommand{\NS}{\mathop{\rm NS}}
\newcommand{\Num}{\mathop{\rm Num}}
\newcommand{\MW}{\mathop{\rm MW}}

\newcommand{\Aut}{\mathop{\rm Aut}}
\newcommand{\Pic}{\mathop{\rm Pic}}

\newcommand{\id}{\mathop{\rm id}}

\newcommand{\Jac}{\mathop{\rm Jac}}

\newtheorem{Theorem}{Theorem}[section]
\newtheorem{Proposition}[Theorem]{Proposition}
\newtheorem{Lemma}[Theorem]{Lemma}

\newtheorem{Corollary}[Theorem]{Corollary}

\newtheorem{tremark}[Theorem]{Remark}
\newenvironment{Remark}{\begin{tremark} \rm}{\end{tremark}}
\newtheorem{Example}[Theorem]{Example}
\newtheorem{tobservation}[Theorem]{Observation}
\newenvironment{Observation}{\begin{tobservation} \rm}{\end{tobservation}}
\newtheorem{tconvention}[Theorem]{Convention}
\newenvironment{Convention}{\begin{tconvention} \rm}{\end{tconvention}}
\newtheorem{tsummary}[Theorem]{Summary}
\newenvironment{Summary}{\begin{tsummary} \rm}{\end{tsummary}}

\begin{document}


\maketitle


\dedication{Dedicated to JongHae Keum on the occasion of his 60th birthday}

\begin{prelims}

\vspace{-0.55cm}

\def\abstractname{Abstract}
\abstract{We classify singular Enriques surfaces in characteristic two
supporting a rank nine configuration of smooth rational curves.
They come in one-dimensional families defined over the prime field,
paralleling the situation in other characteristics, but featuring novel aspects.
Contracting the given rational curves, one can derive
algebraic surfaces
with isolated ADE-singularities and  trivial canonical bundle
whose   $\Q_\ell$-cohomology equals that of a projective plane.
Similar existence results are developed for classical Enriques surfaces.
We also work out an application to integral models of Enriques surfaces (and K3 surfaces).}

\keywords{Enriques surface; characteristic 2; smooth rational curve; elliptic fibration; fake projective plane}

\MSCclass{14J28; 14J27}


\languagesection{Fran\c{c}ais}{%

\vspace{-0.05cm}
\textbf{Titre. Plans projectifs de $\Q_\ell$-cohomologie et surfaces d'Enriques en caract\'e\-ristique 2} \commentskip \textbf{R\'esum\'e.} Nous classifions les surfaces d'Enriques singuli\`eres en caract\'eristique 2 qui supportent une configuration de courbes rationnelles lisses de rang 9. Elles viennent en familles unidimensionnelles d\'efinies sur le corps premier, comme dans la situation des autres caract\'eristiques, mais tout en faisant appara\^{\i}tre de nouveaux aspects. En contractant lesdites courbes rationnelles, on obtient des surfaces alg\'ebriques avec singularit\'es ADE isol\'ees et fibr\'e canonique trivial, dont la cohomologie \`a coefficients dans $\Q_\ell$ co\"{\i}ncide avec celle du plan projectif. Des r\'esultats d'existence similaires sont \'etablis pour les surfaces d'Enriques classiques. Nous mettons \'egalement au point une application aux mod\`eles entiers de surfaces d'Enriques (et de surfaces K3).}

\end{prelims}


\newpage

\setcounter{tocdepth}{1} \tableofcontents

 \section{Introduction}
 \label{s:intro}

 The study of Enriques surfaces forms one of the centerpieces
 for the Enriques--Kodaira classification of algebraic surfaces.
 The subtleties and special properties which they already display over the complex numbers
 are further augmented in characteristic two
 where additional intriguing features arise.
 Despite great recent progress (see \cite{KKM}, \cite{Liedtke}),
 there are still many open problems,
 both on the abstract and the explicit side.
 This paper contributes to both sides
by considering Enriques surfaces
which carry a rank nine configuration
of smooth rational curves.
Contracting the curves, we obtain
a remarkable object:
a singular normal surface with the same \'etale cohomology as projective space $\PP^2$.

  Enriques surfaces admitting such configurations of smooth rational curves
  have previously been studied successfully in \cite{HKO} and \cite{S-Q-hom}
over $\C$ and in \cite{S-Q_l} over fields of odd characteristic.
In characteristic two, however, the theory of Enriques surfaces features many subtleties and surprises
which will also play a lead role in this paper.
Hence it is quite remarkable that some of the techniques from \cite{S-Q_l} can be adapted
 for one essential class of Enriques surfaces in characteristic two,
 namely the so-called singular Enriques surfaces
 whose Picard scheme equals the group scheme $\mu_2$.
 Our first main result parallels those from odd characteristic in quality
 although they are quite different in quantity.

\begin{Theorem}
\label{thm}
There are exactly 19 configurations $R$ of smooth rational curves
of rank $9$
realized on singular Enriques surfaces in characteristic two:
\begin{eqnarray*}
\label{eq:list}
  A_9 ,  A_8 + A_1 ,  A_7 + A_2 ,  A_7 + 2A_1 ,   A_6 + A_2 + A_1 ,    A_5 + A_4 ,  A_5 + A_3 + A_1 , \nonumber \\ \nonumber
  A_5 + 2A_2 , 2A_4 + A_1 ,   3A_3 ,  A_3 + 3A_2 ,
 D_9 ,  D_8 + A_1 ,   D_6 + A_3 ,  \nonumber \\ \nonumber
 D_5 + A_4 ,
 E_8 + A_1 ,  E_7 + A_2 ,  E_6 + A_3 ,  E_6 + A_2 + A_1.
\end{eqnarray*}
For each the following hold:
\begin{enumerate}
\item[(i)]
the root types are supported on 1-dimensional families of Enriques surfaces;
\item[(ii)]
the moduli spaces are irreducible
except for $R=A_8+A_1$ and $A_5+2A_2$ both of which admit two different moduli components;
\item[(iii)]
each family has rational base and is defined over $\F_2$;
\item[(iv)]
each family can be parametrized explicitly,
see Summary \ref{sum2} and Table \ref{T2}.
\end{enumerate}
\end{Theorem}

\begin{Remark}
\label{rem:KK}
Thirteen of the types in Theorem \ref{thm} are supported on singular Enriques  surfaces
with finite automorphism
group in characteristic two, as classified in \cite{KK}, \cite{Martin}
(compare the approach over $\C$ in \cite{HKO}).
This settles the existence part of Theorem \ref{thm} for these types,
but not the stated properties.
\end{Remark}

%
%

%
%

We can also treat the other main class of Enriques surfaces in characteristic two,
the so-called classical Enriques surfaces
 whose Picard scheme equals  $\Z/2\Z$.
 Here we derive the following  existence result:

 \begin{Theorem}
 \label{thm2}
 There are exactly 37 configurations $R$ of smooth rational curves
of rank $9$
realized on classical Enriques surfaces in characteristic two.
More precisely, every rank $9$ root type $R$ embedding
into the even hyperbolic unimodular lattice $U+E_8$ of rank 10
is realized except  for $R=4A_2+A_1$.
 \end{Theorem}

All the root types in question can be found in the first column of Table \ref{T1}.
Theorem \ref{thm2}
largely builds on deep work of Dolgachev--Liedtke (unpublished as of yet)
and Katsura--Kond\=o--Martin \cite{KKM}
which provide a plentitude of constructions and examples
of classical Enriques surfaces with prescribed configurations of smooth rational curves.
At present our methods do not yield enough information
to fully classify the moduli of the classical Enriques surfaces from Theorem \ref{thm2}
(whose  dimensions do in fact vary, largely due to the presence of quasi-elliptic fibrations);
neither do they apply directly to the remaining class of supersingular Enriques surfaces
(compare Remark \ref{rem:ss}).
Yet we can  rule out the root type $R=4A_2+A_1$
and thus prove the following general result:

\begin{Proposition}
\label{prop}
No Enriques surface in any characteristic supports the root type $R=4A_2+A_1$.
\end{Proposition}

Previously, this was proved in characteristic zero in \cite{S-Q-hom},
in odd characteristic in \cite{S-Q_l} and for singular Enriques surfaces as part of
the original proof of Theorem \ref{thm}.
Here we will exhibit a unified geometric argument for all characteristics $\neq 3$ based on triple coverings.

As arithmetic application, we construct integral models of Enriques surfaces (and of K3 surfaces),
i.e.~over integer rings of number fields of small degree (including a few real quadratic fields):

\begin{Theorem}
\label{thm40}
There are  Enriques surfaces over the integer rings of number fields
of  degree $d_0$ supporting maximal root types $R$ as follows:
$$
\begin{array}{c||c|c|c|c|c|c}
R & E_8+A_1 & E_7+A_2  & E_6+A_3 & A_8+A_1& A_5+A_4 & A_6+A_2+A_1     \\
& D_9  && A_9 & & D_5+A_4\\
\hline
d_0 & 2 & 3 & 4 & 6 & 7 & 9\\
\end{array}
$$
\end{Theorem}

In Corollary  \ref{cor:dense}, we also show that integral Enriques surfaces
supporting root types from Theorem \ref{thm40}
(in number fields of fixed degree $>d_0$)
lie dense in the underlying complex moduli spaces.

The paper is organized as follows.
After reviewing some relevant theory around genus one fibrations,
we will right away jump at ruling out many root types
to occur on singular Enriques surfaces (a good portion of them existing in characteristic zero in fact).
The existence and uniqueness part of Theorem \ref{thm}
builds on a rather explicit base change technique
which abstractly goes back to Kond\=o over $\C$ (\cite{Kondo-Enriques})
and is extended to characteristic two in Sections \ref{s:bc}, \ref{s:twist}.
By way of this technique, all Enriques surfaces and root types from Theorem \ref{thm}
fall into two kinds of families which we work out explicitly in the subsequent sections.
In Section \ref{s:classical} we turn to classical Enriques surfaces and prove Theorem \ref{thm2}
before covering Proposition \ref{prop} in Section \ref{s:triple}.
The integral models for Theorem \ref{thm3} and beyond
will be worked out  in Section \ref{s:integral}.
Throughout, we will illustrate the methods and ideas involved
with a series of instructive examples related to the root type $R=A_5+2A_2$
(which is not realized on any singular Enriques surface
with finite automorphism group in characteristic two).


%

\begin{Convention}
All root lattices of type $A_n, D_k, E_l$ are assumed to be negative-definite.
\end{Convention}

\section{Genus one fibrations}
\label{s:elliptic}

Regardless of the characteristic, an Enriques surface can be defined to be a smooth minimal algebraic surface $S$
such that
\[
K_S \equiv 0, \;\;\; b_2(S) = 10
\]
(where the Enriques--Kodaira classification, extended by Bombieri and Mumford in
\cite{BM1}, \cite{BM2}, \cite{Mumford-surf}, enters implicitly).
Here $\equiv$ indicates numerical equivalence and $b_i(S)$ denotes the $i$th Betti number of $S$
for  the $\ell$-adic \'etale cohomology
(with some auxiliary prime $\ell\neq p$).
Moreover, one has
\[
\chi(\mathcal O_S)=1, \;\;\; b_1(S)=0.
\]
From now on we assume that we work over an algebraically closed field $K$ of characteristic $p=2$.
Then Enriques surfaces fall into three cases, depending on their Picard scheme $\Pic^\tau(S)$
which runs through all group schemes of length two
(or, equivalently, depending on the action of the absolute Frobenius morphism on $H^1(S,\mathcal O_S)$
which can be zero- or one-dimensional):
$$
\begin{array}{lcl}
\text{classical:} && \Pic^\tau(S) = \Z/2\Z\\
\text{singular:} && \Pic^\tau(S) = \mu_2\\
\text{supersingular:} && \Pic^\tau(S) = \alpha_2.
\end{array}
$$
All three cases are unified by the way in which they are governed by genus one fibrations,
i.e.~morphisms
\begin{eqnarray}
\label{eq:to}
f: \;\;\; S \to \PP^1
\end{eqnarray}
whose general fiber are curves of arithmetic genus one.
Necessarily, there are multiple fibers involved (of multiplicity two),
but their configuration is genuinely different from the picture in all other characteristics
by \cite[Thms 5.7.5, 5.7.6]{CD}:
\begin{table}[ht!]
\centering
\begin{tabular}{lcl}
classical: && two multiple fibers, both smooth ordinary or of additive type\\
singular: && one multiple fiber, smooth ordinary or of multiplicative type\\
supersingular: && one multiple fiber, smooth supersingular or of additive type.
\end{tabular}
\end{table}

Abstractly, the fibration is encoded in
the lattice
\[
\Num(S) = \Pic(S)/\equiv \;\; \cong U+E_8
\]
where $U$ denote the hyperbolic plane (cf.~\cite{Illusie}). 
Precisely,  the genus one fibration \eqref{eq:to} gives rise to a primitive isotropic effective vector
\[
0< E\in\Num(S) 
\]
by way of the support of a multiple fiber;
by construction, the linear system $|2E|$ has no base points and thus coincides with \eqref{eq:to}.
Conversely, given any non-trivial isotropic vector $E\in\Num(S)$,
it
\begin{itemize}
\item
is effective or anti-effective by Riemann--Roch (so say $E>0$),
\item
can be made primitive by dividing by some appropriate constant, and
\item
one can eliminate the base locus of the linear system $|2E|$
by successive reflections in the classes of smooth rational curves.
\end{itemize}
Recall that the resulting effective isotropic divisor $E'$ is often called a half-pencil or half-fiber.

We emphasize that smooth rational curves are a rather delicate matter on Enriques surfaces
since, unlike on K3 surfaces, $(-2)$-vectors in $\Num(S)$ need not be effective neither anti-effective.
Nonetheless, the above approach will be the key to our findings
since 
we can set it up in such a way
that we retain enough control over the smooth rational curves
provided by the rank $9$ configuration (as explored in \cite{S-Q-hom}, \cite{S-Q_l}).

Instead of going into the details of the gluing techniques and discriminant forms from \cite{S-Q-hom}, \cite{S-Q_l},
we illustrate our methods and ideas with the following instructive example
which will keep on reappearing throughout this paper.
We will use dual vectors following standard notation (cf.~\cite{S-Q-hom}).
E.g., we number the usual generators of the root lattices $A_n$ by $a_1,\hdots, a_n$
and consider the dual vectors $a_i^\vee\in A_n^\vee$ of square $(a_i^\vee)^2=-i(n-i)/n$,
and similarly for $D_n$ and $E_n$.

\begin{Example}[$R=A_5+2A_2$] \rm
\label{ex1}
Assume that the Enriques surface $S$ contains a configuration of smooth rational curves
of type $R=A_5+2A_2$.
Up to isometry,
$R$ admits a unique embedding into $\Num(S)$
with primitive closure
\[
R' = (R\otimes \Q)\cap \Num(S) = E_7 + A_2
\]
and orthogonal complement
\[
R^\perp = \Z H\subset\Num(S), \;\;\; H^2=6.
\]
There are several ways to exhibit isotropic vectors $E\in\Num(S)$.
We highlight two of them:
\begin{enumerate}
\item[\rm (1)]
$a_1^\vee\in A_2^\vee \;\;\;  \Longrightarrow \;\;\; E=a_1^\vee + H/3, \;\;\; E^\perp\cap R = A_5+A_2+A_1.$
\item[\rm (2)]
$e_7^\vee\in E_7^\vee\subset A_5^\vee \;\;\; \Longrightarrow \;\;\; E=e_7^\vee + H/2, \;\;\; E^\perp\cap R = 4A_2.$
\end{enumerate}
Let us point out two properties which will become important momentarily:
in either case, $E$ is primitive (since $E.a_1=1$ resp.~$E.e_7=1$ by construction),
while leaving a large root lattice inside $R$ perpendicular.
Up to changing sign and subtracting the base locus, $|2E|$ will thus induce a genus one fibration on $S$.
\end{Example}

\section{Jacobian vs. extremal genus one fibrations}

It is exactly the approach of Example \ref{ex1}
which we will now pursue in order to get an idea of the Jacobian of  $|2E|$. 
This is always a rational surface endowed with a genus one fibration,
i.e.~either elliptic or quasi-elliptic.
We claim that we can always arrange for $\Jac(f)$ to be rather special,
namely to have finite Mordell--Weil group.
Equivalently, by the Shioda--Tate formula, the fibers support a root lattice of rank $8$.
Following \cite{MP}, such rational elliptic surfaces are called \emph{extremal}.
Note that for quasi-elliptic surfaces, finite Mordell--Weil group is automatic by \cite[\S 4]{RuS}.

\begin{Proposition}
 \label{prop:isotropic}
Let $S$ be an Enriques  surface admitting a configuration $R$ of smooth rational curves of rank $9$.
Then $S$ has a genus one fibration \eqref{eq:to} whose Jacobian is extremal.
\end{Proposition}

\begin{proof}
The proof is the same as in \cite{S-Q_l}.
In short, there are 38 root types $R$ of rank $9$ embedding into $U+E_8$.
For each of them, we exhibit a suitable isotropic vector $E$ as in Example \ref{ex1}
such that $E^\perp\cap R$ is a root lattice $R_0$ of rank $8$ (see Table \ref{T1}).
$R_0$ is not affected by a potential sign change in $E$,
but it may be moved around by the reflections needed to eliminate the base locus of $|2E|$.
In the end, however, $R_0$ will still be supported on $(-2)$-curves, and at the same time on the fibers of $|2E|$
which suffices to conclude the proposition.
\end{proof}

\begin{table}
$$
\begin{array}{|ccccc|}
\hline
R & R' & H^2 & E & R_0\\
\hline
\hline
A_9 & A_9 & 10 & a_2^\vee + 2h/5 & A_7+A_1\\
A_8+A_1 & A_8+A_1 & 18 & (a_3^\vee,0)+h/3 & A_5+A_2+A_1\\
& E_8+A_1 & 2 & (a_3^\vee,0) + h & A_5+A_2+A_1\\
A_7 + A_2 & E_7 + A_2 & 6 & (a_2^\vee,0) + h/2 & A_5+A_2+A_1\\
A_7 + 2A_1 & E_8 + A_1 & 2 & (a_4^\vee,0,0) + h & 2A_3+2A_1\\
A_6+A_2+A_1 & A_6+A_2+A_1 & 42 & (a_1^\vee,0,0)+h/7 & A_5+A_2+A_1\\
A_5 + A_4 & A_5+A_4 & 30 & (0,a_2^\vee)+h/5 & A_5 + A_2 + A_1\\
A_5+A_3+A_1 & E_6+A_3 & 24 & (a_2^\vee,0,0)+h/3 & 2A_3+2A_1\\
A_5+2A_2 & E_7+A_2 &
6 &
(0,0,a_1^\vee)+ h/3 & A_5+A_2+A_1\\
2A_4+A_1 & E_8 + A_1 &
2 &  (0,0,a_1^\vee) + h/2 & 2A_4\\
3A_3 & D_9 &
4 & (a_2^\vee,0,0) + h/2 & 2A_3+2A_1\\

A_3 + 3A_2 & A_3 + E_6 &
 12 & (a_1^\vee,0,0,0) + h/4 & 4A_2 \\
D_9 & D_9 &
4 & d_9^\vee + 3h/4 & A_8\\
D_8+A_1 & E_8+A_1 &
2 & d_8^\vee + h & A_7+A_1\\

D_6+A_3 & D_9 &
4 & (0,a_2^\vee)+h/2 & D_6+2A_1\\

D_5+A_4 & D_5+A_4 &
20 & (d_5^\vee,0) + h/4 & 2A_4\\

E_8+A_1 & E_8+A_1 &
2 & (e_8^\vee,0) + h & E_7+A_1\\
E_7+A_2 & E_7+A_2 &
6 & (e_3^\vee,0) + h & A_5+A_2+A_1\\
E_6+A_3 & E_6+A_3 &
12 & (e_6^\vee,0)+h/3 & D_5+A_3\\
E_6+A_2+A_1 & E_8+A_1 &
2 & (e_1^\vee,0,0)+h & A_5+A_2+A_1\\
\hline
\hline
E_7 + 2A_1 & E_8 + A_1 &
2 & (e_2^\vee,0,0) + h & D_6 + 2A_1 \\

D_7+2A_1 & D_9 &
4 & (d_1^\vee,0,0)+h/2 & D_6+2A_1\\
D_6+A_2+A_1 & E_7 + A_2 &
6 & (0,a_1^\vee,0) + h/3 & D_6+2A_1\\
D_6+3A_1 & E_8+A_1 &
2 & (d_2^\vee,0,0,0) + h & D_4+4A_1\\
D_5+A_3+A_1 & E_8+A_1 &
2 & (d_2^\vee,0,0) + h & 2A_3+2A_1\\
D_5+D_4 & D_9 &
4 & (d_1^\vee,0)+h/2 & 2D_4 \\
D_5+4A_1 & D_9 & 4 & d_1^\vee + h/2 & D_4+4A_1\\
D_4+A_3+2A_1 & D_9 &
4 & (d_1^\vee,0,0,0)+h/2 & 2A_3+2A_1\\
2D_4+A_1 & E_8+A_1 &
2 & (0,0,a_1^\vee)+h/2 & 2D_4\\

D_4+5A_1 & E_8+A_1 & 2 & a_1^\vee + h/2 & D_4+4A_1\\
D_4+A_2+3A_1 & E_7+A_2 & 6 & a_2^\vee+h/3 & D_4+4A_1\\
A_5+A_2+2A_1 & E_8+A_1 &
2 & (0,0,0,a_1^\vee)+h/2 & A_5+A_2+A_1\\
A_4+A_3+2A_1 & A_4+D_5 &
20 & (a_1^\vee,0,0,0)+h/5 & 2A_3+2A_1\\
2A_3+A_2+A_1 & E_7 + A_2 &
6 & (0,0,a_1^\vee,0) + h/3 & 2A_3+2A_1\\
2A_3+3A_1 & E_8+A_1 &
2 & (0,0,0,0,a_1^\vee)+h/2 & 2A_3+2A_1\\
A_3+6A_1 & D_9 & 4 & a_2^\vee + h/2 & 8A_1\\
4A_2+A_1 & E_8+A_1 &
2 & (0,0,0,0,a_1^\vee)+h/2 & 4A_2\\
A_2+7A_1 & A_2+E_7 &
6 & a_2^\vee+h/3 & 8A_1\\
9A_1 & E_8+A_1 & 2 & a_1^\vee + h/2 & 8A_1\\
\hline
\end{array}
$$
\caption{Isotropic vectors and root lattices (Proposition \ref{prop:isotropic})}
\label{T1}
\end{table}

In the next section, we will rule out  half of the above 38 root types on singular Enriques surfaces
(the lower 19 root types from Table \ref{T1}
such that the 19 root types from Theorem \ref{thm} remain).
For this purpose, it will be extremely useful to have the classification of extremal rational elliptic surfaces
in characteristic two handy (due to W. Lang in \cite{Lang1}, \cite{Lang2}).
Before going into the details, we note the following important observation:

\begin{Lemma}
\label{lem:3no}
There are no extremal rational elliptic surfaces
such that the fiber components off the zero section generate the root lattice
\[
2A_3 + 2A_1, \;\;\; D_6+2A_1 \;\;\; \text{ or } \;\;\; 2D_4.
\]
\end{Lemma}

\begin{proof}
Of course, this follows from Lang's classification,
but for the sake of completeness (and since it fits nicely into the scheme of our arguments),
we give a quick proof.
Let $X$ be a rational elliptic surface and denote by $R$ the root lattice generated by the fiber components
perpendicular to the zero section.
The theory of Mordell--Weil lattices \cite{ShMW} provides an equality
\[
\MW(X) = E_8/R.
\]
In the above  cases, this would imply
\[
(\Z/2\Z)^2 \subset \MW(X)
\]
which is clearly impossible in characteristic two.
\end{proof}

The following table lists all the extremal rational elliptic surfaces
in characteristic two  relevant to our issues.
We follow the  notation from \cite{MP},
even though the fiber types often degenerate compared to   $\C$.

%

 \begin{table}[ht!]
 $$
 \begin{array}{|cccc|}
 \hline
 \text{notation} & \text{Weierstrass eqn.} & \text{sing. fibers} & \MW\\
 \hline
 \hline
 X_{9111} & y^2 + xy + t^3y = x^3
 & I_9/0, I_1/t^3+1
 & \Z/3\Z\\
 X_{8211} & y^2 + xy + t^2y = x^3 + t^2x^2
 & I_8/0, III/\infty  & \Z/4\Z\\
 X_{6321} &
 y^2 + xy + t^2y = x^3 &
 I_6/0, IV/\infty, I_2/1 &
 \Z/6\Z\\
  X_{5511} & y^2 + (t+1)xy + t^2 y = x^3+tx^2 & I_5/0,\infty, I_1/t^2+t+1 &  \Z/5\Z\\
 X_{3333} &
 y^2 +xy+(t^3+1)y = x^3
 & I_3/0, t^3+1  & (\Z/3\Z)^2\\
 X_{431} & y^2 + txy + t^2y = x^3  & IV^*/0, I_3/\infty, I_1/1&  \Z/3\Z\\
  X_{321} &
  y^2 + xy = x^3 + tx
  & III^*/\infty, I_2/0 & \Z/2\Z\\
 X_{141} & y^2 + xy = x^3+t^2x
 & I_4/0, I_1^*/\infty  & \Z/4\Z\\
  \hline
 \end{array}
 $$
 \caption{Extremal rational elliptic surfaces in characteristic two}
 \label{T0}
 \end{table}

\begin{Remark}
\label{rem:other}
There are a few further extremal rational elliptic surfaces in characteristic two,
but they will not be relevant to our issues.
\end{Remark}

\section{Non-existence of certain root types on singular Enriques surfaces}
\label{s:latt}

Until Section \ref{s:classical},
we shall now restrict to \textbf{singular Enriques surfaces}.
Then  \eqref{eq:to} always defines an elliptic fibration in the sense
that the general fiber is a smooth curve of genus one (as opposed to quasi-elliptic fibrations),
and the Jacobian $\Jac(f)$ is a rational elliptic surface.

In order to rule out the lower 19 root types from Table \ref{T1},
we have to draw a few more consequences from
the other given data.
Recall that the primitive isotropic vector $E$ is effective, possibly after changing sign
while subtracting the base locus  amounts to a succession $\sigma_0$ of reflections in smooth rational curves;
here $|2\sigma_0(E)|$ induces the elliptic fibration \eqref{eq:to}.
In each case, there is a single $(-2)$-curve $C$ in $R$
such that $E$ is built by adding the dual vector $C^\vee$ and some fraction of $H$.
Hence $\sigma_0(C).\sigma_0(E)=1$, and it follows that $\sigma_0(C)$ comprises
a smooth rational bisection $B$ of the fibration $|2\sigma_0(E)|$
plus possibly some smooth rational curves $C_i$
contained in the fibers of \eqref{eq:to};
again these can be eliminated by reflections,
and since naturally $C_i.\sigma_0(E)=0$, these reflections do not affect $\sigma_0(E)$.

\begin{Summary}
\label{sum1}
There is a composition of reflections $\sigma$ such that $\sigma(E)$ is a half-pencil
and $B=\sigma(C)$ is a smooth rational bisection
while $\sigma(R_0)$ is supported on the fibers and meets $B$ in a prescribed way.
\end{Summary}

We shall now take a closer look at $\sigma(R_0)$.
This root lattice consists of $(-2)$-divisors, each effective or anti-effective
and supported on $(-2)$-curves (the fiber components).
Note that we cannot claim in general that $\sigma(R_0)$ consists of $(-2)$-curves,
and even if that were to hold true, it would not imply that the orthogonal summands $R_v$ of $R_0$
determine the singular fibers of $|2\sigma(E)|$ as extended Dynkin diagrams $\tilde R_v$.
Yet we can utilize the bisection $B$ very much to our advantage.

\begin{Lemma}
\label{lem:<2}
 Assume that $C$ meets less than  two  of the curves forming $R_0=\oplus_v R_v$
 (where all orthogonal summands $R_v$ are irreducible root lattices).
 Then the only possible fiber configuration for $|2\sigma(E)|$ comprises $\tilde R_v$ fibers.
\end{Lemma}

\begin{proof}
This is Criterion 7.1 from \cite{S-Q_l}. The proof carries over literally as it does not depend on the characteristic at all.
\end{proof}

\begin{Proposition}
\label{prop:18no}
The 19 lower root lattices in Table \ref{T1}
cannot be supported on singular Enriques surfaces in characteristic two.
\end{Proposition}

\begin{proof}
We apply Lemma \ref{lem:<2} to the data from Table \ref{T1}.
Except for $R=A_3+6A_1$, it follows that $R_0$ determines the singular fibers of $|2\sigma(E)|$.
But then, except for $R=A_5+A_2+2A_1$ and $4A_2+A_1$,
neither case is compatible with Lang's classification
(see especially Lemma \ref{lem:3no}).
Both exceptional root types will be ruled out later:
$R=A_5+A_2+2A_1$ using the base change construction in \ref{ss:no},
$R=4A_2+A_1$
based on a unified geometric argument in Section \ref{s:triple}.

It remains to exclude the root type $R=A_3+6A_1$
(which over $\C$ is not compatible with the orbifold Bogomolov--Miyaoka--Yau inequality).
The two $(-2)$-curves  $a_1, a_3\subset A_3$
which meet $C=a_2$ are taken to $(-2)$-divisors $\sigma(a_1), \sigma(a_3)$
whose support could potentially involve two different simple  components of a single fiber $F_0$
both of which are met by the bisection $B=\sigma(a_2)$.
All other summands $A_1$'s from $R$ are either mapped into $F_0$,
or they give rise to fibers of type $\tilde A_1$
(which follows as in the proof of Lemma \ref{lem:<2}).
A priori, this leads to the configurations
\[
\tilde D_4 + 4\tilde A_1, \;\; 2\tilde D_4, \;\; \tilde D_6 + 2\tilde A_1, \;\; \tilde E_7+\tilde A_1
\]
(no $\tilde E_8$ since it only has one simple fiber component),
but by inspection of Lang's classification (and Lemma \ref{lem:3no}),
this only leaves $\tilde E_7+\tilde A_1$.
By some reflections in fiber components, we can arrange for $a_1, a_3$ to map to
the two simple fiber components met by $B$.
But then the 5 orthogonal copies of $A_1$ embed into the fiber minus these two components,
and by orthogonality, minus the adjacent components, i.e.~into $D_4$
which is  impossible for rank reasons.
%
%
\end{proof}

We illustrate these ideas by continuing to investigate the root typ $R=A_5+2A_2$
as 
in Example \ref{ex1}.

\begin{Example}[$R=A_5+2A_2$ cont'd] \rm
\label{ex2}
Consider the root type $R=A_5+2A_2$ with the second isotropic vector from Example \ref{ex1}.
The fiber configurations  accommodating $R_0=4A_2$ are
\begin{eqnarray}
\label{eq:4A2}
4\tilde A_2, \;\;\; \tilde E_6+\tilde A_2, \;\;\; \tilde E_8.
\end{eqnarray}
However, the last configuration is not compatible with the bisection meeting some fiber component with multiplicity one,
but the other two work perfectly fine.
\end{Example}

For completeness, we record the following nice consequence of Proposition \ref{prop:18no}:

\begin{Corollary}
The orbifold Bogomolov--Miyaoka--Yau inequality holds for any $\Q_\ell$-cohomology projective plane
obtained from a singular Enriques surface in characteristic two.
\end{Corollary}

\section{Base change construction}
\label{s:bc}

To translate the  data from Table \ref{T1} into geometric information,
and eventually into explicit equations,
we elaborate on a base change construction
which was originally due to Kond\=o in  \cite{Kondo-Enriques}
and later extended in  \cite{HS}.
Notably, a singular Enriques surface $S$ in characteristic two
still has a universal K3 cover
\[
X \to S
\]
which exhibits $S$ as a quotient of $X$ by a fixed point free involution $\tau$
(unlike for classical and supersingular Enriques surfaces
where $X$ may only be K3-like etc).
Note that $X$ inherits an elliptic fibration from $S$
(for instance by pulling back the half-pencil $\sigma(E)$),
\begin{eqnarray}
\label{eq:K3}
X \to \PP^1,
\end{eqnarray}
which presently is endowed with the two disjoint sections $O, P$
which the smooth rational bisection $B$ splits into.
The next result translates Kond\=o's work \cite{Kondo-Enriques} over $\C$ to characteristic two,
combining the Enriques involution $\tau$ with  translation by $P$, denoted by $t_P$.
It has been obtained independently by Martin in \cite{Martin}
in relation with Enriques surfaces with finite automorphism groups.

%
%
%

\begin{Proposition}
\label{prop:imath}
There is an involution $\imath$ on $X$
such that
\[
\tau = t_P \circ \imath.
\]
\end{Proposition}

\begin{proof}
The proof follows the line of arguments in \cite{S-Q-hom}
although we have to make several modifications in order to account
for the special features in characteristic two.
Consider the automorphism
\[
\imath = t_P^{-1}\circ\tau \in \Aut(X).
\]
Then $\imath$ fixes $O$ as a set, with a single fixed point in the ramified fiber, say $F_\infty$.
We now turn to $\imath^2$ which the proposition states to be the identity.
Since $\imath^2$ fixes
fibers as sets and
$O$ pointwise,
it acts as an automorphism on the generic fiber $X_\eta$ of \eqref{eq:K3}.
In particular,
\[
\mbox{ord}(\imath^2)\in\{1,2,3,4,6\}.
\]
If $3\mid\mbox{ord}(\imath^2)$, then the induced action of $\imath^2$ on the  regular $1$-form of $X_\eta$
involves a primitive cube root of unity.
Naturally, this extends to the action on a regular $2$-form $\omega$ on $X$.
On the other hand, the involution $\tau^*$ leaves $\omega$ invariant, and so does $(t_P)^*$,
so we obtain a contradiction.

It remains to rule out the cases $\mbox{ord}(\imath^2)\in\{2,4\}$.
We switch to the ramified fiber $F_\infty$.
If $F_\infty$ is smooth, then $\imath\in\Aut(F_\infty)$
(since $\imath$ fixes $O|_{F_\infty}$),
and it follows right away that
$\mbox{ord}(\imath)=4, \;
\mbox{ord}(\imath^2)=2$
and
$F_\infty$ is supersingular,
contradicting the classification of ramified fibers reviewed in Section \ref{s:elliptic}.

If $F_\infty$ is singular, i.e.~of Kodaira type $I_{2n}$ for some $n\geq 1$,
then $\imath$ acts as an automorphism of the smooth locus $\Theta_0^\#$ of the identity component $\Theta_0$
of $F_\infty$. But this leads to the multiplicative group:
\begin{eqnarray}
\label{eq:0}
\Theta_0^\# \cong \mathbb G_m \;\; \Longrightarrow \;\;  \Aut(\Theta_0^\#) = \Z/2\Z \;\; \Longrightarrow \;\;
(\imath^2)|_{\Theta_0^\#}=\id.
\end{eqnarray}
However, with $\imath^2$ acting as inversion on the generic fiber,
it cannot act as identity on any component of a multiplicative fiber
(for instance, inversion interchanges the nodes where the fiber components intersect).
This contradiction completes the proof of Proposition \ref{prop:imath}.
\end{proof}

%

\begin{Corollary}
\label{cor:RES}
The involution $\imath$ leads us back to the rational elliptic surface $\Jac(f)$:
\[
X/\imath=\Jac(f).
\]
\end{Corollary}

\begin{proof}
As the involution interchanges the unramified fibers and fixes $O$ as a set,
the quotient $X/\imath$ inherits an elliptic fibration with section induced from $O$.
It remains to study the ramified fiber $F_\infty$.

If $F_\infty$ is smooth, then the proof of Proposition \ref{prop:imath} shows
that $\imath$ restricts to an automorphism of $F_\infty$,
with the following two possibilities:
\[
\imath|_{F_\infty} = \pm \id.
\]
In fact, it is easy to see that $\imath|_{F_\infty} = \id$, since otherwise $\tau$ would fix
any point $Q_\infty\in F_\infty$ with $2Q_\infty=P|_{F_\infty}$.
But then the Euler-Poincar\'e characteristic reveals that $X/\imath$ is a rational elliptic surface,
and since it has exactly the same fibers as the Enriques surface $Y$
(including the half-pencil equalling $F_\infty$),
we deduce that $X/\imath= \Jac(f)$ as stated.

If $F_\infty$ is non-smooth, i.e.~multiplicative of type $I_{2n} \;\; (n>0)$,
then we number the components cyclically as usual, starting from the zero component $\Theta_0$ met by $O$
up to $\Theta_{2n-1}$.
Again we have seen in the proof of Proposition \ref{prop:imath} that $\imath$ restricts to an automorphism
of the smooth locus $\Theta_0^\#$ of $\Theta_0$;
more precisely, by \eqref{eq:0}, $\imath|_{\Theta_0^\#}$ either is the identity or inversion in $\Theta_0^\#\cong\mathbb G_m$.
Note that for $\tau$ to be fixed point free, $P$ intersects a different component than $\Theta_0$
so that $t_P$ induces a rotation of the fiber components.
Anyway, if
\[
\imath|_{\Theta_0^\#}\neq \id,
\]
then one checks that  $\tau$ either fixes a component
(as a set, but then it contains a fixed point), or it interchanges two adjacent components
(such that the intersection point is fixed).
Hence we infer that
\[
\imath|_{\Theta_0^\#}=\id,
\]
so that all fiber components are fixed by $\imath$ as sets,
and $P$ has to meet $\Theta_n$ for $\tau$ to have order two.
Subsequently, this implies that $X/\imath$ attains a fiber of type $I_n$, and we conclude as before.
\end{proof}

The corollary explains the title of this section:
$X$ is a quadratic base change of the Jacobian of the elliptic fibration \eqref{eq:to}
on the Enriques surface $S$.
It is instructive to distinguish
whether $P$ is two-torsion or not,
since in the first case, it will descend to $\Jac(f)$ (as it is invariant under the action of $\imath$)
while in the second case it is honestly anti-invariant for the action of $\imath$ by definition.

\begin{Example}[$R=A_5+2A_2$ cont'd] \rm
\label{ex3}
We illustrate the two cases above with possible configurations supporting the root type $R=A_5+2A_2$.
Arguing with the first isotropic vector from Example \ref{ex1},
Lemma \ref{lem:<2} implies that $\Jac(f)=X_{6321}$.
For the 2-torsion case, we find the following  configuration of $(-2)$-curves
with ramified $I_6$ fiber:
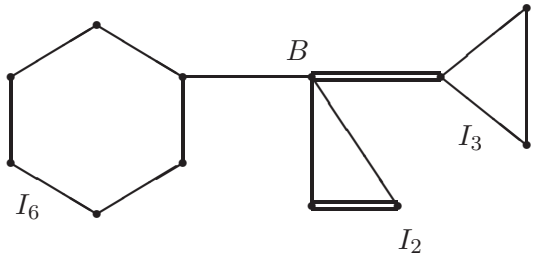
\begin{figure}[ht!]
\setlength{\unitlength}{.45in}
\begin{picture}(8,2.9)(3,0)
\thicklines


\multiput(4,2)(2,0){2}{\circle*{.1}}
\multiput(4,1)(2,0){2}{\circle*{.1}}
\put(4,2){\line(0,-1){1}}
\put(6,2){\line(0,-1){1}}
\put(5,0.4){\circle*{.1}}
\put(5,2.6){\circle*{.1}}


\put(4,2){\line(5,3){1}}
\put(5,0.4){\line(5,3){1}}
\put(4,1){\line(5,-3){1}}
\put(5,2.6){\line(5,-3){1}}

%

\put(7.5,2){\circle*{.1}}

\put(7.5,2){\line(-1,0){1.5}}


\put(9,2){\circle*{.1}}
\put(7.5,2.04){\line(1,0){1.5}}
\put(7.5,1.96){\line(1,0){1.5}}

\put(10,2.8){\circle*{.1}}
\put(10,1.2){\circle*{.1}}

\put(9,2){\line(5,4){1}}
\put(9,2){\line(5,-4){1}}

\put(10,2.8){\line(0,-1){1.6}}


\put(7.5,.5){\circle*{.1}}
\put(8.5,.5){\circle*{.1}}

\put(7.5,.46){\line(1,0){1}}
\put(7.5,.54){\line(1,0){1}}

\put(7.5,2){\line(0,-1){1.5}}
\put(7.5,2){\line(2,-3){1}}

%
\put(7.2,2.2){$B$}
\put(8.5,0){$I_2$}
\put(9.05,1.2){$IV$}
\put(4.05,.4){$I_6$}
%
\thinlines
%
%
\end{picture}
\caption{Configuration with multiple $I_6$ fiber supporting the root type $R=A_5+2A_2$}
\label{Fig:R}
\end{figure}

If  the smooth rational bisection $B$ does not produce a  two-torsion section, the picture looks as follows
(with a dashed line indicating that the respective fiber could also be ramified):
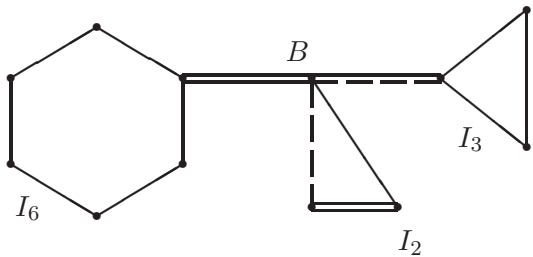
\begin{figure}[ht!]
\setlength{\unitlength}{.45in}
\begin{picture}(8,2.9)(3,0)
\thicklines


\multiput(4,2)(2,0){2}{\circle*{.1}}
\multiput(4,1)(2,0){2}{\circle*{.1}}
\put(4,2){\line(0,-1){1}}
\put(6,2){\line(0,-1){1}}
\put(5,0.4){\circle*{.1}}
\put(5,2.6){\circle*{.1}}


\put(4,2){\line(5,3){1}}
\put(5,0.4){\line(5,3){1}}
\put(4,1){\line(5,-3){1}}
\put(5,2.6){\line(5,-3){1}}

%

\put(7.5,2){\circle*{.1}}

\put(7.5,2.04){\line(-1,0){1.5}}
\put(7.5,1.96){\line(-1,0){1.5}}


\put(9,2){\circle*{.1}}
\put(7.5,2.04){\line(1,0){1.5}}
\put(7.5,1.96){\line(1,0){1.5}}

\put(10,2.8){\circle*{.1}}
\put(10,1.2){\circle*{.1}}

\put(9,2){\line(5,4){1}}
\put(9,2){\line(5,-4){1}}

\put(10,2.8){\line(0,-1){1.6}}


\put(7.5,.5){\circle*{.1}}
\put(8.5,.5){\circle*{.1}}

\put(7.5,.46){\line(1,0){1}}
\put(7.5,.54){\line(1,0){1}}

\multiput(7.5,1.96)(0,-.4){4}{\line(0,-1){.3}}
\put(7.5,2){\line(2,-3){1}}

%
\put(7.2,2.2){$B$}
\put(8.5,0){$I_2$}
\put(9.05,1.2){$IV$}
\put(4.05,.4){$I_6$}
%
\thinlines
%
%
\end{picture}
\caption{Configurations supporting the root type $R=A_5+2A_2$ without multiple $I_6$ fiber}
\label{Fig:2}
\end{figure}

\end{Example}

\subsection{Non-existence of $R=A_5+A_2+2A_1$}
\label{ss:no}

Assume there is a singular Enriques surface $S$ supporting the root type $R=A_5+A_2+2A_1$.
Consider the isotropic vector $E$ from Table \ref{T1}.
By Lemma \ref{lem:<2}, it induces a genus one fibration with $\Jac(f)=X_{6321}$.
The smooth rational bisection $B$ obtained from the last orthogonal summand $A_1$ in $R$
meets the singular
fibers as follows (where dashed lines indicate that the fibre may be multiple):

\begin{figure}[H]
\setlength{\unitlength}{.45in}
\begin{picture}(8,2.9)(3,0)
\thicklines


\multiput(4,2)(2,0){2}{\circle*{.1}}
\multiput(4,1)(2,0){2}{\circle*{.1}}
\put(4,2){\line(0,-1){1}}
\put(6,2){\line(0,-1){1}}
\put(5,0.4){\circle*{.1}}
\put(5,2.6){\circle*{.1}}


\put(4,2){\line(5,3){1}}
\put(5,0.4){\line(5,3){1}}
\put(4,1){\line(5,-3){1}}
\put(5,2.6){\line(5,-3){1}}

%

\put(7.5,2){\circle*{.1}}

\put(7.5,2.04){\line(-1,0){1.5}}
\multiput(7.5,1.96)(-.4,0){4}{\line(-1,0){.3}}


\put(9,2){\circle*{.1}}
\put(7.5,2.04){\line(1,0){1.5}}
\put(7.5,1.96){\line(1,0){1.5}}

\put(10,2.8){\circle*{.1}}
\put(10,1.2){\circle*{.1}}

\put(9,2){\line(5,4){1}}
\put(9,2){\line(5,-4){1}}

\put(10,2.8){\line(0,-1){1.6}}


\put(7.5,.5){\circle*{.1}}
\put(8.5,.5){\circle*{.1}}

\put(7.5,.46){\line(1,0){1}}
\put(7.5,.54){\line(1,0){1}}

\put(7.46,2){\line(0,-1){1.5}}
\multiput(7.54,2)(0,-.4){4}{\line(0,-1){.3}}

%
\put(7.2,2.2){$B$}
\put(8.5,0){$I_2$}
\put(9.05,1.2){$IV$}
\put(4.05,.4){$I_6$}
%
\thinlines
%
%
\end{picture}
\end{figure}

On the K3 cover $X$, the bisection splits into disjoint sections $O, P$ meeting
always the same fibre component unless the fibre is ramified (such that opposite components
are met).
Depending on the ramification, we will now work out a primitive sublattice of $\NS(X)$
which will prove useful in establishing a contradiction to our assumption.

If no singular fibre ramifies,
then $P$ has height $h(P)=4$, and we obtain an orthogonal decomposition
\begin{eqnarray}
\label{eq:N_1}
\NS(X) \supseteq N_1 = U + (2A_5+2A_3+2A_1)' + \langle -4\rangle
\end{eqnarray}
where the primitive closure of the root lattice in the middle is an index 6 overlattice
of rank $16$ and determinant $36$
generated by the class of one of the 6-torsion sections, say $Q$.
If $N_1$ were not primitive, then  $P+nQ$ would be $2$-divisible for some $n\in\{0,\hdots,5\}$.
Here $n=2, 4$ are equivalent to the case $n=0$
while odd $n$ can be ruled out since then $P+nQ$ would meet the $I_2$ fibers
at the non-identity component, thus not allowing for 2-divisibility.
Hence we may assume that $P=2P_0$
where $P_0.O=0$ and $h(P_0)=1$ by construction.
By \cite{ShMW}, the height implies that $P_0$ meets both $I_6$ fibers at
the component opposite to that met by $O$.
Since this is met by the two-torsion section $3Q$ as well,
we obtain the height pairing $\langle P_0, 3Q\rangle \leq 2 - 2\cdot\frac 32 \leq -1$,
giving the desired contradiction (since the height pairing with a torsion section always evaluates as zero).

If the $I_2$ fibre ramifies, then $P$ meets the $I_4$ fiber on $X$ at the opposite component, so $h(P)=3$,
and we obtain a rank $20$ sublattice $N_2\subseteq \NS(X)$ of determinant $108$
by \cite[(11.22)]{SSh}.
If this were not primitive, then $P$ would have to be two-divisible as above,
but then we obtain a contradiction from
\[
h(P_0) = 4 - \frac 34 - \begin{cases} 0\\\frac 32\end{cases} - \begin{cases} 0\\\frac 32\end{cases}
\neq \frac 34 = \frac 14\, h(P).
\]

Finally, if the $I_6$ fibre were to ramify, then $P$ would meet the opposite component of $I_{12}$
on $X$,
and $\langle P,3Q\rangle\leq -1$ as before.

In summary, $X$ always comes with a primitive embedding of $N_1$ or $N_2$ into $\NS(X)$.
Note that no matter what field $S$ and $X$ are defined on,
we can always specialize to the situation where they are defined over some finite field $\F_q$
(which we  increase, if necessary, to ensure that $\NS(X)$ is defined over $\F_q$).
Then $X$ is ordinary by \cite[Cor.~A.2]{KK},
and since the Tate conjecture holds for $X$ by \cite{ASD}, we infer $\rho(X)=20$.
Therefore the Artin--Tate conjecture\footnote{which is equivalent to the Tate conjecture by \cite{Milne};
for characteristic 2, this requires input from \cite{LLR}
which proves that not only in odd characteristic the size of the Brauer group is a square,
so that Milne's original argument also goes through in characteristic two.}
implies that $\det(\NS(X))$ is odd.
In the $N_2$ case, this gives a contradiction right away;
for $N_1$ one can argue with the $2$-length of the discriminant group $N_1^\vee/N_1$.
Presently this is at least $2$ (by the two orthogonal summands with even determinant in \eqref{eq:N_1}),
so $\NS(X)$ will have $2$-length at least $1$ (by the same kind of argument as for \cite[Thm. 6.1]{KS}
or \cite[Prop. 3.3]{RS}).
That is, $\det(\NS(X))$ is even, contradiction.
Thus there is no  singular Enriques surface
supporting the root type $R=A_5+A_2+2A_1$, as stated in Theorem \ref{thm}.

\section{Quadratic twist}
\label{s:twist}

We shall now extract conditions from the data developed in Section \ref{s:bc} and from Table \ref{T1}
and start to convert them into explicit equations towards the proof of Theorem \ref{thm}.
Continuing from Summary \ref{sum1}, we proceed as follows:
\begin{itemize}
\item
use Lemma \ref{lem:<2} to single out the possible fiber configurations
supporting a given rank $9$ root type $R$ (as in Example \ref{ex3});
\item
prove that there are reflections in fiber components mapping all curves from $R_0$
to fiber components (and not affecting the smooth rational bisection $B$); this amounts to a case-by-case analysis
paralleling \cite[\S 7]{S-Q-hom}
using standard properties of root lattices and Weyl groups
(cf.~the proof of Lemma \ref{lem:<2} for the argument for the root type $R=A_3+6A_1$).
\end{itemize}
Especially the second step allows us to predict exactly how the bisection $B$ intersects the singular fibers
(depending on their multiplicity, so there may be a few cases to distinguish as in Example \ref{ex3}).
Note that this directly carries over to information on the section $P$ on the K3 cover $X$.
In particular, we can determine whether $P$ is two-torsion of not.
We now turn to the problem of exhibiting explicit equations for the Enriques surfaces in question.

\subsection{Two-torsion case}
\label{ss:2t}

If the configuration determines $P$ as a two-torsion section,
then there is little left to do:
there always is a ramified reducible fiber (necessarily of multiplicative type),
so $X$ arises from $\Jac(f)$ by a quadratic base change ramified at this fiber.
Hence such K3 surfaces (and their quotient singular Enriques surfaces) occur in one-dimensional families
(depending on the free parameter of the base change).

\begin{Example}[$R=A_5+2A_2$ cont'd] \rm
\label{ex4}
In the two-torsion case from Example \ref{ex3}, there is ramification at the $I_6$ fiber,
so the base changes can be normalized to take the shape
\begin{eqnarray}
\label{eq:bcn}
t\mapsto \lambda s^2/(s-1) \;\;\;\; (\lambda \in K^\times).
\end{eqnarray}
By inspection of Figure \ref{Fig:R}, the resulting singular Enriques surfaces also support the root types $E_7+A_2$ and $A_7+A_2$.
\end{Example}

\begin{Summary}
\label{sum2}
We obtain three one-dimensional families of singular Enriques surfaces supporting the following root types:
$$
\begin{array}{rcc}
R \;\;\;\;\;\;  & \Jac(f) & \text{mult.~fiber}\\
\hline
E_8+A_1, D_8+A_1, A_8+A_1, A_7+2A_1 & X_{321} & I_2\\
E_7+A_2, A_7+A_2, A_5+2A_2 & X_{6321} & I_6\\
E_6+A_2+A_1 & X_{6321} & I_2\\
\end{array}
$$
\end{Summary}

\subsection{Non-torsion case}
\label{ss:non}

Here the section $P$ is really anti-invariant for the induced action of the deck transformation $\imath$
-- or equivalently, invariant for the composition $\jmath = \imath \circ (-\id)$
where $(-\id)$ denotes the involution of the generic fiber.
Hence $P$ descends to a section $P'$ on the minimal resolution of the quotient surface,
\[
X' = \widetilde{X/\jmath}.
\]
This is again an elliptic  K3 surface with the same fibers at $\Jac(f)$ except at the ramified fiber
(which generically is of Kodaira type $I_4^*$, the additional components accounting for the isolated fixed points
of $\jmath$ in the ramified fiber, cf.~the proof of Corollary \ref{cor:RES}).
$X'$ is often called  \emph{quadratic twist} of $\Jac(f)$.
In order to compute explicit equations, we arrange for the quadratic base change to be ramified at $\infty$
and thus take the shape
\begin{eqnarray}
\label{eq:bc}
s\mapsto s(s+1).
\end{eqnarray}
To this end, we fix a suitable (reducible) fiber at $t=0$ and move around the other fibers.
Starting from the models in Table \ref{T0}, this can be achieved by way of the M\"obius transformation
\begin{eqnarray}
\label{eq:ml}
t \mapsto \frac t{\mu t+\lambda} \;\;\;\; (\lambda\neq 0),
\end{eqnarray}
for instance, in agreement with the general fact that quadratic base changes of a given rational elliptic surface
come in two-dimensional families.
So let us assume that the rational elliptic surface $\Jac(f)$ has Weierstrass equation
\[
\Jac(f): \;\;\; y^2 + a_1xy + a_3y = x^3 + a_2x^2 + a_4x+a_6, \;\;\; a_i\in K[t], \;\; \deg(a_i)\leq i,
\]
with singular fiber moving around according to the M\"obius transformation \eqref{eq:ml}.
If $X$ arises from $\Jac(f)$ by the quadratic base change \eqref{eq:bc},
then one can directly compute the invariants for the involution $\jmath\in\Aut(X)$.
They result in the following Weierstrass equation of the quadratic twist $X'$:
\begin{eqnarray}
\label{eq:X'}
X': \;\;
y^2 + a_1xy + a_3y = x^3 + a_2x^2 + a_4x+a_6 + t(a_1x+a_3)^2.
\end{eqnarray}
Enter the section $P\in\MW(X)$ which ought be disjoint from $O$;
this property is often called integral and equivalent to $P=(U,V)$
for polynomials $U,V\in K[t]$ of degree at most $4$ resp.~$6$.
For $P'$, this translates as $P'=(U', V')$ of degree at most $2$ resp.~$3$.
There is an important additional condition for the top degree coefficient of $U'$:

\begin{Observation}
\label{obs:U'}
Write out the top degree coefficients
\[
U' = u_2t^2+\hdots, \;\;\;\; a_1 = a_{11}t+{a_{10}}, \;\;\;\; a_3 = a_{33}t^3+\hdots,
\]
depending on $\mu,\lambda$. Then $u_2=a_{33}/a_{11}$.
\end{Observation}

\begin{proof}
With the above degree conditions on $U', V'$, it is clear that $a_{11}u_2$ and $a_{33}$ are the only degree $7$ coefficients
of the equation
%
obtained from  \eqref{eq:X'} upon substituting $P'$.
Hence solving for the equation to vanish identically leads to the above relation.
\end{proof}

\begin{Remark}
\label{rem:a_1}
Observation \ref{obs:U'} also implies that $a_1\not\equiv 0$.
\end{Remark}

Observation \ref{obs:U'} also ensures that $P'$ meets a far simple component of the (generic) $I_4^*$ fiber of $X'$ at $\infty$
-- and leads to a degree $6$ polynomial equation in $t$ as explained in the proof.
Then we either try to solve directly for this equation to vanish identically
or throw in some additional information, for instance about fibers met non-trivially by $P'$
(inferred from the intersection pattern of the smooth rational bisection $B$
with the singular fibers  on the Enriques surface $S$, or from torsion sections present on $X$
as we shall exploit in the next sections).

\section{Explicit equations for root type  $R=A_5+2A_2$}
\label{s:eqns}
\label{s:ex}

We illustrate the approach in the non-torsion case as sketched in \ref{ss:non}
by elaborating on our usual exemplary root type $R=A_5+2A_2$.
Consider an Enriques surface $S$ with elliptic fibration \eqref{eq:to}  induced by the first isotropic vector in Example \ref{ex1}.
Recall that $\Jac(f)=X_{6321}$ with the possible fiber configurations  from Example \ref{ex3}.
Since we have already settled the case of a multiple $I_6$ fiber in Example \ref{ex4}, we shall assume that the $I_6$ fiber
is unramified.
Thus it is safe to locate it at $t=0$ as in Table \ref{T0}.

For starters, we restrict to the case where the fiber of type $I_2$ is unramified ($\mu\neq 1$ in \eqref{eq:ml}).
Presently, we could set out to calculate the section $P'$ directly on the quadratic twist $X'$,
using the condition that it meets the $I_2$-fiber non-trivially,
but it turns out to be even more beneficial to work on the K3 cover $X$ itself
and just remember that $P$ is induced from $X'$.
In particular, this implies that $U$ is invariant under the deck transformation
\begin{eqnarray}
\label{eq:deck}
\imath: \;\;\; s\mapsto s+1
\end{eqnarray}
of \eqref{eq:bc}
while $V$ decidedly is not, for otherwise $P$ would be induced from $\Jac(f)$.

We shall facilitate that $X$ admits a 3-torsion section at $Q=(0,0)$,
precisely~it takes the general shape
\begin{eqnarray}
\label{eq:3t}
X: \;\;\; y^2 + a_1xy + a_3y = x^3
\end{eqnarray}
(with reducible fibers at the zeroes of $a_3$, presently of Kodaira types $I_6$ and $IV$).
It follows from divisibility considerations in $K[s]$,
or more general from the theory of Mordell--Weil lattices \cite{ShMW},
that $P.Q=2$. By the given shape of $Q$, this implies that $U$ and $V$ share a common factor $g$ of degree two.
Since $P$ meets the fibers of type $I_6$ and $IV$ at the identity component,
$g$ is relatively prime to $a_3$, and we infer from vanishing orders that in fact $g^3\mid V$.
Recalling that $V$ is not invariant under $\imath$, we deduce that
\[
g(s) \neq g(s+1) = h(s).
\]
On the other hand, the invariance of $U$ under $\imath$ leads, after absorbing some factor into $g$ if necessary, to
\[
U = gh, \;\;\; V = \nu g^3 \;\;\; (\nu \in K^\times).
\]
But then comparing the substitution into \eqref{eq:3t} with the relation obtained from $\imath^*P=-P$,
we read off $\nu=1$.
We write out $g=g_2s^2+g_1s+g_0$ (with $g_2+g_1\neq 0$, for otherwise $g=h$)
and solve for the substitution into \eqref{eq:3t}
\begin{eqnarray}
\label{eq:gh}
g^3 + a_1gh+ a_3 = h^3
\end{eqnarray}
to vanish identically, taking into account additionally that the $I_2$ fibers are met non-trivially.
The top coefficients of \eqref{eq:gh} directly give $g_2=1$ and $g_1=1/\mu$.
Then the $I_2$ fiber condition implies $\lambda = g_0\mu(g_0\mu+\mu+1)$
whence
\[
g_0^2\mu^3+g_0\mu^2(\mu+1)+(\mu+1)^2=0.
\]
This rational curve is parametrized by
\[
g_0 = w(\mu+1)/\mu, \;\;\; \mu = 1/(w^2+w).
\]
All in all, this gives rise to a one-dimensional family of singular Enriques surfaces
(also displayed in Table \ref{T2} after eliminating the symmetry $u=w^2+w$),
supporting the root type $A_5+2A_2$ as in Figure~\ref{Fig:2}.

If the $I_2$ fiber were to ramify (i.e.~$\mu=1$), then the same approach would go through,
locating the fiber at $\infty$ and
yielding $g_2=g_1=1$ and thus $g=h$, contradiction.
This completes the analysis of the singular Enriques surfaces
supporting the root type $R=A_5+2A_2$.

\begin{Remark}
A similar analysis can be carried out starting from the second isotropic vector from Example \ref{ex1}.
In fact, the corresponding systems of equations are straightforward to solve,
thanks to the fibers met non-trivially.
In case $\Jac(f)=X_{431}$, we directly obtain
\[
U'=\mu s^2 \;\;\; \text{ and } \;\;\;
V' = \lambda s^3 \;\;\; \text{ (up to exchanging $P'$ and $-P'$)}
\]
which subsequently leads to a one-dimensional family parametrized by $\lambda=\sqrt\mu^3$.
For $\Jac(f)=X_{3333}$, however,  the resulting equations
are a little more complicated to display.

In either case, the  families of Enriques surfaces resulting from both approaches are easily related.
This can be achieved by singling out the multiple $I_6$ fiber in the diagram
of $(-2)$-curves underlying the configuration
from Example \ref{ex3}
(or a $IV^*$ fiber with smooth rational bisection and disjoint $A_2$ in Figure \ref{Fig:R}),
or as part of a more general pattern explored in the context
of Enriques surfaces with four cusps in \cite{RS} (over $\C$,
but the arguments carry over to characteristics $\neq 3$).
\end{Remark}

\section{One-dimensional families}

Having treated the exemplary root type $R=A_5+2A_2$  in full detail,
we shall now state the main classification result needed to prove Theorem \ref{thm}.

 \begin{Theorem}
 \label{thm1}
 Let $S$ be a singular Enriques surface in characteristic two supporting a root lattice $R$ of rank $9$.
 Consider the section $P$ on the K3-cover of $S$ obtained from the data in Table \ref{T1}.
 If $P$ is not two-torsion,
then $S$ and $R$ appear in Table \ref{T2}.
 \end{Theorem}

 Table  \ref{T2}
 involves the elliptic fibration \eqref{eq:to} induced by the isotropic vector from Table \ref{T1}
 (or its Jacobian), and
 the $x$-coordinate $U'$ of the section $P'\in\MW(X')$ for the quadratic twist $X'$
 which in turn is determined by the parameters $\mu, \lambda$ entering
in \eqref{eq:ml}.
Of course, we always have to exclude a few values for $(\mu,\lambda)$
where the Enriques surfaces degenerate, but we omit the details for brevity.

By default, we usually start from the Weierstrass form of $\Jac(f)$ given in Table \ref{T1}
though in one instance, the equations look much nicer
when starting from the other affine standard coordinate $v=1/t$ of $\PP^1$
(to which we then apply \eqref{eq:bc} and  \eqref{eq:ml} analogously).

 \begin{table}[ht!]
 $$
 \begin{array}{|lccc|}
 \hline
 \text{root type} & {\Jac(f)} & \text{section:} \, U' & \text{quadr.~twist:} \, (\mu, \lambda)\\
 \hline
 A_9 & X_{8211} & t(t +1/\mu^2) & (\mu,1/\mu)\\
 A_8 + A_1 & X_{6321} & t(t+(\mu+1)/\mu^2) &
(\mu,(\mu+1)/\mu)\\

A_6+A_2+A_1 & X_{6321} &
t(t+(\mu+1)/\mu^2) &
(\mu,1/(\mu(\mu+1)))\\

A_5+A_4 & X_{6321} &
\mu^2v(v+\mu^2+\mu) &
(\mu,\mu^3)\\

A_5+A_3+A_1 &
X_{141} &
\text{ as } E_6+A_3 &\\
%
 A_5+2A_2
&
X_{6321} &
{t^2 + u(u+1)t+u(u+1)^2}
&
(1/u, (u+1)^2/u)
\\
 2A_4+A_1 & X_{5511} &
\multicolumn{2}{l|}{(u^2+u+1)^2t^2+(u^4+u+1)(u+1)^2u^2t+(u^4+u+1)^2(u+1)^2u^2}
\\
&&\multicolumn{2}{r|}
{
\left(\frac{(u^2+u+1)^2}{(u+1)^2u^2}, \frac{(u^4+u+1)^3}{(u+1)^4u^4}\right)
}
\\
 3A_3 & X_{9111}& \text{ as } D_9
& \\

A_3+3A_2
& X_{3333} & ((\lambda^6+1)t^2+\lambda^3t+\lambda^6)/\lambda^4
& (1/\lambda^2,\lambda)
\\
D_9 & X_{9111} & \lambda^2t^2 &
(1/\lambda^2, \lambda)
\\

 D_6+A_3 & X_{222} & \text{ as } D_9 &\\

 D_5+A_4 & X_{5511} & u^2t^2/(u+1)^2
 & (u^2, u^5/(u+1))
 \\
 E_6+A_3 & X_{141} &
 \mu^2t
&(\mu,\mu^2)
\\
 \hline
 \end{array}
 $$
 \caption{Singular Enriques surfaces for 13 maximal root types} 
 \label{T2}
 \end{table}

 \subsection{Proof of Theorem \ref{thm1}}

 Since we have treated one case in full detail
 and all others follow the same line of argument,
 we omit the details of the proof of Theorem \ref{thm1} for space reasons
 (except for the  case missing from the proof of Proposition \ref{prop:18no},
 to be covered separately in Section \ref{s:triple}).

 \begin{Remark}
 Some moduli components from
 characteristic zero cease to exist in characteristic two,
 even though the root type itself may still be supported on some singular Enriques surfaces.
 Those components are ruled out by the direct calculations,
 or alternatively, by more structural arguments as those involved in the proof of Proposition \ref{prop:18no}.
 We illustrate this by the following example.
 \end{Remark}

 \begin{Example}[$R=D_8+A_1$] \rm
 For the root type $R=D_8+A_1$, the isotropic vector from Table \ref{T1}
 leads to $R_0=A_7+A_1$, embedding into the fiber configurations
 \[
 \tilde A_7+\tilde A_1, \;\;\; \tilde E_7+\tilde A_1, \;\;\; \tilde E_8.
 \]
 While the last configuration is not compatible with a smooth rational bisection meeting some simple fiber component
with multiplicity one, the first configuration can only be ruled out by direct computation.
Alternatively, consider the isotropic vector $E'=d_2^\vee+h$.
This has $R_0'=D_6+2A_1$.
Using Lemma \ref{lem:3no} and the above argument for $\tilde E_8$,
one immediately derives the fiber configuration $\tilde E_7+\tilde A_1$,
and then a height argument shows that there is a ramified $I_2$ fiber
and a two-torsion section involved, i.e.~we are in the first family of Summary \ref{sum2}.
 \end{Example}

 \begin{Remark}
Some computer algebra systems experience surprising difficulties
 in characteristic two calculations (notably factorization of polynomials,
 but also basic simplifications);
 fortunately, the present problem always provides a sanity check
 when verifying that the computed section $P'$ indeed lies on the quadratic twist $X'$.
 \end{Remark}

 \begin{Remark}
Similar ideas can be applied to study singular Enriques surfaces
with finite automorphism group (see recent work of Martin \cite{Martin}).
\end{Remark}

%
%
%
%

\section{Moduli components}
\label{s:comp}

With Theorem \ref{thm1} and Summary \ref{sum2} at our disposal,
the proof of Theorem \ref{thm} is almost complete.
This section proves statement \textit{(ii)} about the moduli components
for root types
$$R=A_8+A_1, \;\;\; A_5+2A_2.$$
This amounts to verifying that the families
exhibited in Table \ref{T2} resp.~Summary \ref{sum2}
are indeed distinct.
We shall prove the following slightly stronger statement:

\begin{Lemma}
\label{lem:comp}
Let $R=A_8+A_1$ or $A_5+2A_2$.
There are two distinct families of K3 covers of singular Enriques surfaces supporting $R$.
\end{Lemma}

\begin{proof}
Compared to what had to be done in \cite{S-Q_l}, especially in characteristic $3$,
the arguments are greatly simplified thanks to the  fact
that a singular Enriques surface in characteristic two cannot have a supersingular K3 surface as universal cover;
that is,
\[
\rho(X)\leq 20
\]
(and, in fact, the height is one by \cite[Cor.~A.2]{KK}).
In turn, this implies that a generic member $X_\eta$ of either of the one-dimensional families of K3 surfaces at hand
has
\[
\rho(X_\eta)=19
\]
(since it is provided with plenty of smooth rational curves from the Enriques surface).
We continue by comparing the discriminants $d$ of $\NS(X_\eta)$ using \cite[(11.22)]{SSh}.
For $R=A_5+2A_2$, this returns once $d=12$ and the other time $d=108$.
To see this, it suffices to check that there cannot be further torsion sections than the present cyclic group
$\Z/6\Z$,
and that for the second family, the section $P$ mapping to the smooth rational bisection $B\subset Y$ generically
has height   $h(P)=3$.
At the same time, neither $P$ nor its translates by torsion sections
may be $2$-divisible since then the height would result in $3/4$
which is not compatible with the contraction terms in the height pairing.
Similarly, if $P$ were $3$-divisible, i.e.~$P=3Q$ then $h(Q)=1/3$
could a priori be accommodated by certain intersection patterns,
but each case would lead to a contradiction by calculating that the height pairing
with some two- or six-torsion section would be negative
(as opposed to being zero).
Finally, translating $P$ by some torsion section does not make a difference
as the two-torsion section itself is $3$-divisible and the sections of higher torsion order
cause the $I_3$ fibers to be met non-trivially,
so that there cannot be any $3$-divisibility at all.

The argument for $R=A_8+A_1$ is similar, so we leave the details to the reader.
\end{proof}

\begin{Remark}
Alternatively, one could argue with Nikulin's root invariant to conclude that
the given families of Enriques surfaces are distinct (cf.\ \cite{Nikulin}).
For the families from Summary \ref{sum2},
this returns $(E_8+A_1, \{0\})$ resp. $(E_7+A_2, \{0\})$;
the corresponding families in Table \ref{T2} have root invariants
$(A_8+A_1, \Z/3\Z)$ resp. $(A_5+2A_2, \Z/3\Z)$.
\end{Remark}

\section{Classical Enriques surfaces}
\label{s:classical}

In this section, we turn our attention to the objects of Theorem \ref{thm2},
namely \textbf{classical Enriques surfaces} $S$.
Here the striking difference is that the universal cover \eqref{eq:K3} of $S$ is no longer smooth,
but it is K3-like in the sense that
its dualizing sheaf is trivial.
Although $X$ need not even be normal,
generally it has 12 isolated $A_1$ singularities only (cf.~\cite{EHSB}).

\begin{Example} \rm
\label{ex:12A1}
The latter situation  persists if $S$ admits an elliptic fibration without additive fibers;
more precisely, in this situation there are only singular fibers of Kodaira type $I_n$,
with indices adding up to the Euler--Poincar\'e characteristic $12$,
and $X$ inherits a fibration with the same fibers, but with surface singularities
of type $A_1$  at the 12 nodes of the fibers.
The minimal resolution
\begin{eqnarray}
\label{eq:Xtilde}
\tilde X \to X
\end{eqnarray}
gives a supersingular K3 surface.
\end{Example}

In contrast, if $S$ admits a quasi-elliptic fibration (another central case for our considerations to follow),
then $X$ always is non-normal as it is singular along the cuspidal curve.

The proof of Theorem \ref{thm2} proceeds in three steps:
\begin{enumerate}
\item[\rm (1)]
we realize a number of root types on Enriques surfaces with finite automorphism group
as classified in \cite{KKM};
\item[\rm (2)]
we realize almost all other root types on Enriques surfaces admitting certain quasi-elliptic fibrations
following Dolgachev--Liedtke;
\item[\rm (3)]
we rule out the remaining root type $R=4A_2+A_1$ by adapting the techniques
for singular Enriques surfaces from the previous sections.
\end{enumerate}

\subsection{Maximal root types on Enriques surfaces with finite automorphism group}
\label{ss:finite}

This section follows closely in spirit the approach from \cite{HKO}
to realize maximal root types on Enriques surfaces with finite automorphism group
-- only that the classification for classical Enriques surfaces in characteristic two
by Katsura--Kond\=o--Martin in \cite{KKM}
looks rather different then Kond\=o's original classification over $\C$ in \cite{Kondo-Enriques}
(cf.\ Remark \ref{rem:KK}).
To this end, the following table lists the maximal root types $R$ together with the classical Enriques surfaces $S$
with finite automorphism group,
in the notation from \cite{KKM}, which support them.
Of course, the Enriques surfaces $S$ need not be unique
as we are merely concerned with the existence of $R$.

$$
\begin{array}{|c|c||c|c|}
\hline
R & \text{ type of } S &
R & \text{ type of } S\\
\hline
A_9 & \tilde E_8 & D_8+A_1 & \tilde E_7+\tilde A_1\\
A_8+A_1 & \tilde E_8 &
D_7+2A_1 & \tilde E_7+\tilde A_1\\
A_7+A_2 & \tilde E_7+\tilde A_1 &
D_6+A_3 & \tilde E_7+\tilde A_1\\
A_7+2A_1 & \tilde E_6+\tilde A_2 &
D_6+A_2+A_1 & \tilde D_8\\
A_6+A_2+A_1 & \tilde E_8 &
D_6+3A_1 & \tilde D_4+\tilde D_4\\
A_5+A_4 & \tilde E_8 &
D_5+A_4 & \tilde E_8\\
A_5+A_3+A_1 & \tilde E_7+\tilde A_1 &
D_5+A_3+A_1 & {\rm VIII}\\
A_5+2A_2 & \tilde E_6+\tilde A_2 &
D_5+4A_1 & {\rm VIII}\\
A_5+A_2+2A_1 & {\rm VIII}
&
D_5+D_4 & \tilde D_8\\
A_4+A_3+2A_1 & \tilde E_7+\tilde A_1
&
E_8+A_1 & \tilde E_8\\
3A_3 & \tilde E_6+\tilde A_2 &
E_7+A_2 & \tilde E_8\\
2A_3+A_2+A_1 & \tilde E_6+\tilde A_2 &
E_7+2A_1 & \tilde D_8\\
A_3+3A_2 & \tilde E_6+\tilde A_2 &
E_6+A_3 & \tilde E_8\\
A_3+6A_1 & \tilde D_4+\tilde D_4 &
E_6+A_2+A_1 & \tilde E_6+\tilde A_2\\
D_9 & \tilde E_8 &&\\
\hline
\end{array}
$$

Along the same lines, one can easily verify that the root type
$R=2A_4+A_1$ is supported on the one-dimensional family of classical and supersingular Enriques surfaces
constructed by Katsura and Kond\=o in \cite{KK}.
In total this allows us to realize 30 maximal root types on classical Enriques surfaces.

\begin{Remark}
\label{rem:ss}
Some of the above types of Enriques surfaces with finite automorphism groups
also admit supersingular realizations (thus supporting $R$), but unfortunately not all of them.
It is unclear to us whether the corresponding root types
may be realized on supersingular Enriques surfaces with infinite automorphism group.
\end{Remark}

\subsection{Maximal root types on quasi-elliptic fibrations}
\label{ss:qe}

Given a genus one fibration \eqref{eq:to} on an Enriques surface $S$,
we can characterize $S$ as a torsor over the Jacobian of \eqref{eq:to}.
However, this construction usually gives no control over the $(-2)$-curves on $S$
-- except for those contained in the fibers (which is why we had to make such an effort
to work out the maximal root types on singular Enriques surfaces in the preceding sections
(and outside characteristic $2$ in \cite{S-Q-hom}, \cite{S-Q_l})).

An exceptional case consists in elliptic surfaces with a quasi-elliptic fibration
since here the cuspidal curve naturally provides another $(-2)$-curve
(which in fact is a bisection).
Hence one can arrange for the ramification of the fibers to accommodate certain root types
--
but only on classical Enriques surfaces, since for supersingular Enriques surfaces
it is unclear how to control the underlying torsors.
We learned this approach from I. Dolgachev and C. Liedtke
who apply it to construct certain classical Enriques surfaces with finite automorphism group,
and in particular with crystallographic root lattices.

Quasi-elliptic rational surfaces have been classified by Ito in \cite{Ito} according to the configuration of reducible fibers.
There are 7 cases, but we will only need two of them, with fiber configuration
\[
\tilde D_6+2\tilde A_1, \;\;\; 2\tilde D_4.
\]
On the torsor, if a fiber is not ramified, then the cuspidal curve $C$ automatically meets
\begin{itemize}
\item
the node of $\tilde A_1$ (which has Kodaira type $III$, of course);
\item
the double component of $\tilde D_4$;
\item
the central double component of $\tilde D_6$ (by symmetry,
for instance enforced by the automorphisms induced from two-torsion sections on the Jacobian).
\end{itemize}
In comparison, ramification forces $C$ to meet some simple component transversally.
This rough information suffices to realize several maximal root types $R$.
They are indicated in the following table, together with the fiber configuration and the ramified reducible fibers (if any).

$$
\begin{array}{|c|c|c|}
\hline
R & \text{fiber configuration} & \text{ ramified reducible fibers}\\
\hline
2A_3+3A_1 & \tilde D_6+2\tilde A_1 & 2\tilde A_1\\
A_2+7A_1 & 2\tilde D_4 & \tilde D_4\\
9A_1 & 2\tilde D_4 & -\\
2D_4+A_1 & 2\tilde D_4 & 2\tilde D_4\\
D_4+A_3+2A_1 & \tilde D_6+2\tilde A_1 & \tilde D_6+\tilde A_1\\
D_4+A_2+3A_1 & 2\tilde D_4 & 2\tilde D_4\\
D_4+5A_1 & 2\tilde D_4 & \tilde D_4\\
\hline
\end{array}
$$

\subsection{Proof of existence part of Theorem \ref{thm2}}

The existence statements
of Theorem \ref{thm2}
are coverered by our findings in \ref{ss:finite} and \ref{ss:qe}
(but only for classical Enriques surfaces).

\section{Non-existence of root type $R=4A_2+A_1$}
\label{s:triple}

This section gives a unified proof of Proposition \ref{prop}
for Enriques surfaces over fields of any characteristic $\neq 3$,
thus also completing the proof of Theorem \ref{thm} and of Theorem \ref{thm2}.
We pursue an approach using triple covers in the spirit of \cite{KS}.

Assume that $S$ is an Enriques surface over an algebraically closed field $k$
of characteristic $\neq 3$
which supports the root lattice $R=4A_2+A_1$.
Let the $A_2$ summands be generated by $(-2)$-curves $C_i, C_i'$ for $i=1,\hdots,4$
(so $C_i^2=C_i'^2=-2, C_i.C_i'=1$ and all other intersection numbers are zero).
Recall from Table \ref{T1} that $R$ has primitive closure
\[
R'=E_8+A_1\subset\Num(S).
\]
This implies that, up to interchanging some $C_i, C_i'$,
the divisor
\[
D = \sum_{i=1}^3 (C_i-C_i')
\]
is $3$-divisible in $\Num(S)$
(but the root $D/3$ is neither effective nor anti-effective by \cite{S-nodal}).
By \cite{Miranda}, this induces a triple covering of $S$.
Explicitly, let $P_i=C_i\cap C_i'$ and consider the blow-up
$$\pi: \;\; \tilde S\to S$$
of $S$ in the three points $P_1, P_2, P_3$
with exceptional curves $E_i$.
Then the strict transforms $\tilde C_i, \tilde C_i'$ are $(-3)$-curves,
and the pull-back
\[
\tilde D = \pi^*D= \sum_{i=1}^3 (\tilde C_i-\tilde C_i')
\]
is $3$-divisible in $\Num(\tilde S)$.
It follows from the general theory of triple coverings \cite{Miranda}
that $\tilde D$ gives rise to a smooth triple covering
\[
\tilde Y \to \tilde S
\]
with branch locus the support of $\tilde D$.
Blowing down successively the pre-images of $\tilde C_i, \tilde C_i'$ and $E_i$
($i=1,2,3$),
we arrive at a smooth algebraic surface $Y$.
As in \cite[5.1.3]{KS}, one computes that $K_Y\equiv 0$ and $\chi(Y)=12$
(the Euler--Poincar\'e characteristic in $\ell$-adic \'etale cohomology, $\ell\neq$ char$(k)$).
By the classification of algebraic surfaces, it follows that $Y$ is an Enriques surface.

To conclude, consider the curves $C_4, C_4'$ and the generator of the orthogonal summand $A_1$ of $R$.
Then each of these curves is tripled in $\tilde Y$; more precisely its pre-image consists of three disjoint $(-2)$-curves
(since otherwise there would be a branch point or a singularity).
That is, $\tilde Y$ comes equipped with the root lattice $\tilde R=3A_2+3A_1$ supported on $(-2)$-curves
which maps down directly to $Y$.
However, $\tilde R$ does not embed into $\Num(Y)=U+E_8$.
This gives the desired contradiction and proves Proposition \ref{prop} uniformly outside characteristic $3$
(while the case of characteristic $3$ was covered among others in \cite{S-Q_l}).
\hfill $\Box$

\subsection{Conclusion}

Proposition \ref{prop} proves the final missing parts of Theorem \ref{thm}
and Theorem \ref{thm2}..
\hfill $\Box$

\section{Integral models}
\label{s:integral}

We conclude with an application to integral models of Enriques surfaces (and of K3 surfaces).
To this end, we start with the following observation combining Theorem \ref{thm} and \cite{S-Q_l}:

\begin{Corollary}
\label{cor9}
Exactly 9 of the families of Enriques surfaces supporting maximal root types
occur in every characteristic.
\end{Corollary}

In detail, this amounts to the first two families from Summary \ref{sum2} and the first 5 as well as the last 4 families from
Table \ref{T2} (thus also including root type $3A_3$).
In total, there are 16 maximal root types supported on these families.
Naturally this leads to the  problems of exhibiting the families over $\Z$,
and ideally even members over $\Z$, or over integer rings of number fields of small degree.
We start by discussing the first point.

\begin{Theorem}
\label{prop:int-fam}
\label{thm3}
All 9 families above admit models over $\Z$,
i.e.~they may be parametrized by a univariate polynomial ring over $\Z$.
\end{Theorem}

\begin{Remark}
For two of the families, integral models also appear in \cite{Martin}.
\end{Remark}

\begin{proof}
The proof of Theorem \ref{prop:int-fam} can be achieved by explicit construction.
The key property which makes all of this work,
is that many rational elliptic surfaces admit models over $\Z$;
here the fiber types may degenerate, but the underlying Dynkin types may not.
For instance, among the extremal rational elliptic surfaces from Table \ref{T0},
only $X_{3333}$ does not admit an integral model (it degenerates in characteristic $3$
due to the full $3$-torsion).

Given an isotropic vector $E$ from Table \ref{T1}
such that the Jacobian of $|2E|$ has an integral model,
it thus remains to set up the base change construction over $\Z$.
We achieve this by interpolating between the models from this paper and from \cite{S-Q_l}.
As before, we distinguish two cases.

If the configuration determines the section $P$ to be two-torsion (as in \ref{ss:2t}),
then the appropriate base change is immediate:
For the first family from Summary \ref{sum2},
the Jacobian $X_{321}$ has integral model exactly as in Table \ref{T0}
(with $I_1$ fiber at $t=1/64$ outside characteristic $2$).
Hence the base change from \eqref{eq:bcn} suffices perfectly.
For the second family, the same applies to the integral model of $X_{6321}$
given by
\[
X_{6321}: \;\;\;
y^2 + (2t-1)xy - t(t-1) y = x^3 + tx^2.
\]
This has singular fibers of type $I_6$ at $t=0$, $I_3$ at $\infty$,
$I_2$ at $t=1$ and $I_1$ at $t=-1/8$,
except that
$I_3$ and $I_1$ degenerate to type $IV$ in characteristic $2$
and
$I_2$ and $I_1$ degenerate to type $III$ in characteristic $3$
(which is why the third family from Summary \ref{sum2} does not admit an integral model).

The non-torsion case is a little more complicated to handle
since now the section $P$ depends on the very base change
-- not only its equations, but also the existence of $P$.
For shortness, we concentrate on a single family, say
the configuration of type $A_9$.
By Table \ref{T1}, we ought to start with an integral model of $X_{8211}$,
but following the approach laid out in  \ref{ss:non},
we  squeeze in an (a priori superfluous) parameter $\mu\in K^\times$
so that we can take the fixed quadratic base change from \eqref{eq:bc}:
\[
X_{8211}: \;\;\;
y^2 +(\mu t+1)xy  -\mu(\mu t+1) t^2 y
=
x^3 - \mu t^2 x^2.
\]
This has  singular fibers of Kodaira types $ I_8$ at $t=0$, $ I_2$ at $t=-1/\mu$,
and $ I_1$ at the zeroes of
$(\mu^2+16\mu) t^2+2\mu t+1$,
except in characteristic two where the last 3 fibers degenerate to a single ${III}$ fiber.
The  K3 surfaces $Y$ resulting from the base change \eqref{eq:bc} are endowed with the section
\[
P = (s(s+1)(\mu s^2+\mu s+1), \;s^2(s+1)(\mu s^2+\mu s+1))
\]
which is anti-invariant for the deck transformation $s\mapsto -1-s$.
Hence $\tau$ indeed defines the required fixed point free involution on $Y$
(compatible with the families from Table \ref{T2} and from \cite{S-Q_l}).

The remaining 6 families work similarly;
the details are left to the reader.
\end{proof}

Turning to the second problem alluded to above,
we emphasize that it is unclear whether Enriques surfaces may admit models over $\Z$
(although there has been substantial progress on the problem
achieved recently by Liedtke and Martin).
Indeed we will fall short of exhibiting models over $\Z$,
but our approach works over small number fields:

\begin{Theorem}[= Theorem \ref{thm40}]
\label{thm4}
Within the 9 families, there are  Enriques surfaces over the integer rings of number fields
of  degree $d_0$ as follows:
$$
\begin{array}{c||c|c|c|c|c|c}
R & E_8+A_1 & E_7+A_2  & E_6+A_3 & A_8+A_1& A_5+A_4 & A_6+A_2+A_1     \\
& D_9  && A_9 & & D_5+A_4\\
\hline
d_0 & 2 & 3 & 4 & 6 & 7 & 9\\
\end{array}
$$
\end{Theorem}

\begin{proof}
Since our Enriques surfaces live in families over $\Z$,
it suffices to control the ramification of the base change in order to prevent
\begin{itemize}
\item
the singular fibers from degenerating (mostly from becoming multiple), 
\item
the involution $\tau$ from attaining fixed points.
\end{itemize}
In the $2$-torsion cases, there is a multiple singular fiber at zero anyway,
the other ramification point of the base change \eqref{eq:bcn} being $4\lambda$ (outside characteristic two).
We shall ensure that this point does not hit any singular fiber.
For the first family from Summary \ref{sum2}, this translates as
\[
 4 \lambda \neq 0, \frac 1{64}
\]
at every place coprime to $2$ simultaneously.
That is, both $\lambda$ and $256\lambda-1$ are units in every local ring
(including those above $2$).
Globally, there is a strikingly simple solution:
Just postulate that
\begin{eqnarray}
\label{eq256}
\lambda = \frac 1\mu \;\;\; \text{ and } \;\;\;
\mu(\mu-256)=\pm 1,
\end{eqnarray}
pinning down $\lambda$ uniquely up to conjugation in one out of two real quadratic fields.
For the second family from Summary \ref{sum2}, the analogous reasoning leads to the claimed
degree three extensions of $\Q$ given by
\begin{eqnarray}
\label{eq2nd}
\mu(\mu-4)(\mu+32)= \pm 1 \;\;\;\;\;\; (\lambda=1/\mu).
\end{eqnarray}

We continue by discussing the $A_9$ configuration with integral model
exhibited in the proof of Theorem \ref{thm3}.
The base change \eqref{eq:bc} is ramified at $\infty$ and $-1/4$ (outside characteristic two).
The fiber at $\infty$ is smooth for $\mu\neq 0, -16$
while the fiber at $-1/4$ degenerates exactly when $\mu=\pm 4$.
As before, this leads to integer units in the degree 4 extensions of $\Q$ encoded in
\begin{eqnarray}
\label{eq3rd}
\mu(\mu+4)(\mu-4)(\mu+16)=\pm 1.
\end{eqnarray}
All other families work similarly and are thus omitted for brevity.
\end{proof}

\begin{Remark}
One can also verify that the degrees in Theorem \ref{thm4} are optimal, depending on the given root type.
\end{Remark}

\begin{Remark}
Incidentally, we also obtain integral models of the K3 covers. Their arithmetic will be studied elsewhere.
\end{Remark}

We conclude this paper with a little observation
concerning integral points inside the moduli spaces.

\begin{Corollary}
\label{cor:dense}
Pick one of the 9 families from Corollary \ref{cor9}
and let $d_0$ as in Theorem \ref{thm4}.
For fixed $d>d_0$, the Enriques surfaces over the integer rings of all degree $d$ number fields
lie dense inside the family.
\end{Corollary}

\begin{proof}
Nothing prevents us from multiplying the left-hand side of
the central equations  \eqref{eq256}, \eqref{eq2nd}, \eqref{eq3rd} by some monic  polynomial $g\in\Z[\mu]$
of degree $d-d_0$.
By Hilbert's irreducibility theorem, there will be infinitely many $g\in\Z[\mu]$ such that the resulting polynomial is irreducible,
each providing a Galois orbit of Enriques surfaces over the integer rings of degree $d$ number fields.
Since the complex moduli of our families are one-dimensional by \cite{S-Q-hom}
(and in fact in any characteristic by Theorem \ref{thm} and \cite{S-Q_l}),
this suffices to prove the density.
\end{proof}

\subsection*{Acknowledgements}

Many thanks to Igor Dolgachev,
Toshiyuki Katsura,
Christian Liedtke and Gebhard Martin for  helpful comments
and sharing their insights.
I'm grateful to the referee for her/his suggestions
improving the paper.

\bibliographymark{References}

\providecommand{\bysame}{\leavevmode\hbox to3em{\hrulefill}\thinspace}
\providecommand{\arXiv}[2][]{\href{https://arxiv.org/abs/#2}{arXiv:#1#2}}
\providecommand{\MR}{\relax\ifhmode\unskip\space\fi MR }
\providecommand{\MRhref}[2]{%
  \href{http://www.ams.org/mathscinet-getitem?mr=#1}{#2}
}
\providecommand{\href}[2]{#2}

\end{document}